\documentclass{amsart}
\usepackage{graphicx}
\usepackage{amsmath}
\theoremstyle{plain}
\begingroup
 
\newtheorem{theorem}{Theorem} 
\newtheorem{proposition}{Proposition}
\newtheorem{corollary}{Corollary}
\newtheorem{lemma}{Lemma}
\endgroup
\newcommand{\nwc}{\newcommand}
\nwc{\qref}[1]{(\ref{#1})}
\nwc{\cadlag}{c\`{a}dl\`{a}g}
\nwc{\la}{\label}
\nwc{\nn}{\nonumber}
\nwc{\ee}{\text{e}}
\nwc{\Z}{\mathbb{Z}}
\nwc{\C}{\mathbb{C}}
\nwc{\E}{\mathbb{E}}
\nwc{\R}{\mathbb{R}}
\nwc{\N}{\mathbb{N}}
\nwc{\Mn}{\mathbb{M}_N}
\nwc{\PP}{\mathcal{P}}
\nwc{\sub}{\mathcal{S}}
\nwc{\M}{\mathcal{M}}
\nwc{\law}{\stackrel{\mathcal{L}}{\rightarrow}}
\nwc{\eqd}{\stackrel{\mathcal{L}}{=}}
\nwc{\vp}{\varphi}
\nwc{\Vp}{\Phi}
\nwc{\ve}{\varepsilon}
\nwc{\veps}{\varepsilon}
\nwc{\eps}{\ve}
\nwc{\dnto}{\downarrow}
\nwc{\nsup}{^{(n)}}
\nwc{\ksup}{^{(k)}}
\nwc{\jsup}{^{(j)}}
\nwc{\nksup}{^{(n_k)}}
\nwc{\inv}{^{-1}}
\nwc{\argmin}{\mathrm{argmin}}
\nwc{\argmax}{\mathrm{argmax}}
\nwc{\Stoch}{\mathfrak{Q}}
\nwc{\ad}{\mathrm{ad}}
\nwc{\Ad}{\mathrm{Ad}}
\nwc{\diag}{\mathrm{diag}}
\nwc{\dd}{\mathfrak{d}}
\nwc{\gln}{\mathfrak{g}}
\nwc{\unn}{\mathfrak{u}}
\nwc{\lotri}{\mathfrak{l}}
\nwc{\GLM}{GL(m,\C)}
\nwc{\proj}{P}
\nwc{\Tr}{\mathrm{Tr}}
\nwc{\Id}{\mathrm{Id}}
\nwc{\son}{\mathfrak{so}(n)}
\nwc{\so}{\mathfrak{so}}
\nwc{\SO}{\mathsf{SO}}
\nwc{\SOn}{\mathsf{SO}(n)}
\nwc{\Symm}{\mathrm{Symm}}
\nwc{\SymmN}{\mathsf{Symm}(N)}
\nwc{\MM}{\mathcal{M}}
\nwc{\TT}{\mathbb{T}}
\nwc{\Rq}{\mathbb{R}^q}
\nwc{\GL}{\mathsf{GL}}
\nwc{\schwartz}{\mathcal{S}}
\nwc{\laplacian}{\triangle}
\nwc{\eigen}{\psi}
\nwc{\one}{\mathbf{1}}
\nwc{\id}{I}
\nwc{\Lap}{\mathcal{L}}
\nwc{\Sz}{Sz\'{e}kelyhidi}
\nwc{\Poincare}{Poincar\'{e}}
\nwc{\Mobius}{M\"{o}bius}
\nwc{\Arnold}{Arnol'{d}}
\nwc{\Muller}{M\"{u}ller}
\nwc{\Sverak}{\v{S}ver\'{a}k}
\nwc{\Berard}{B\'{e}rard}
\nwc{\Ampere}{Amp\`{e}re}
\nwc{\Scal}{\mathrm{Scal}}
\nwc{\Ric}{\mathrm{Ric}}
\nwc{\nnn}{\mathbf{n}}
\nwc{\bb}{\mathbf{b}}
\nwc{\Holder}{H\"{o}lder}
\nwc{\Levy}{L\'{e}vy}
\nwc{\T}{\mathbb{T}}
\nwc{\Varphi}{\Phi}
\nwc{\Ito}{It\^{o}}
\nwc{\vol}{\mathrm{vol}}
\nwc{\Vol}{\mathrm{Vol}}
\nwc{\rhoeq}{\rho_{\mathrm{eq}}}
\nwc{\matn}{\mathbb{M}}
\nwc{\grad}{\mathrm{grad}}
\nwc{\Lie}{\mathcal{L}}
\nwc{\Hess}{\mathrm{Hess}}
\nwc{\res}{\varepsilon}
\nwc{\ressq}{\tilde{\res}}
\nwc{\HH}{\mathcal{H}}
\nwc{\hern}{\mathbb{H}(n)}
\nwc{\On}{\mathbb{O}(n)}
\nwc{\DD}{\mathcal{D}}
\newcommand{\RN}[1]{%
  \textup{\uppercase\expandafter{\romannumeral#1}}%
}

\theoremstyle{definition}
\newtheorem{definition}[lemma]{Definition} 
\newtheorem{remark}[theorem]{Remark}

\theoremstyle{remark}
\newtheorem{notation}{Notation} 
 
\numberwithin{equation}{section}
\numberwithin{figure}{section}

\newcommand{\ud}{\mathrm{d}}
\newcommand{\bd}{\boldsymbol}
\newcommand{\ii}{\ensuremath{\boldsymbol{i}}}

\begin{document}
\title{Siegel Brownian motion}
\author{Govind Menon}
\author{Tianmin Yu}
\address{Division of Applied Mathematics, Brown University, 182 George St., Providence, RI 02912. }

\email{govind\_menon@brown.edu}
\email{tianmin\_yu@brown.edu}

\thanks{This work has been supported in part by the National Science Foundation (DMS 2107205). GM also acknowledges partial support from JSPS KAKENHI grants JP18H01124 and JP20K20884 and expresses his gratitude to the Research Institute for Mathematical Sciences at Kyoto University and the Department of Mathematics, Kyushu University during the completion of this work.}

\subjclass[2010]{58G32,60D05,30.0X}
\keywords{Brownian motion on Riemannian manifolds, Dyson Brownian motion, Symplectic geometry, Symmetric space}

\begin{abstract}
We construct an analogue of Dyson Brownian motion in the Siegel half-space $\HH$ that we term {\em Siegel Brownian motion\/}. Given $\beta \in (0,\infty]$, a stochastic flow for $Z_t\in \HH$ is introduced so that the law of the eigenvalues $\lambda_t$ of the cross ratio matrix $\mathfrak R(Z_t,\ii I_n)$ is determined, after a change of variables to $\sigma(\lambda) \in (0,\infty)^n$, 
by the \Ito\/ differential equation
\begin{equation}
\label{eq:siegel-eig3} 
\ud \sigma^k_t=\frac12\left(\coth{\sigma^k_t}+\sum_{l\neq k}\frac{\sinh{\sigma^k_t}}{\cosh{\sigma^k_t}-\cosh{\sigma^l_t}}\right)\ud t+\sqrt{\frac2\beta}\ud W^k_t, \quad k=1,\ldots, n,
\end{equation}
where $W_t$ is a standard Wiener process in $\R^n$.

This interacting particle system corresponds to stochastic gradient ascent 
\begin{equation}
  \label{eq:evol-sigma-b}
        \ud \sigma_t= \frac{1}{2}\nabla S(\sigma_t) +\sqrt{\frac2\beta}\ud W_t, 
\end{equation}    
where $S(\sigma)= \log \mathrm{vol}\,\mathcal{O}_{\lambda(\sigma)}$ is a Boltzmann entropy that enumerates the microstates in the group orbit 
$\mathcal{O}_\lambda = \{Z \in \HH\left| \mathrm{eig}\left(\mathfrak R(Z,\ii I_n)\right)=\lambda \right. \}$. In the limit $\beta=\infty$, the group orbits $\mathcal{O}_{\lambda_t}$ evolve by motion by minus a half times mean curvature.

\end{abstract}

\maketitle

\section{Introduction}
\label{sec:intro}
\subsection{Outline}
In a recent paper, the first author along with Huang and Inauen, introduced a geometric construction of Dyson Brownian motion~\cite{HIM22}. This paper uses this framework to derive an analogue of Dyson Brownian motion in a symmetric space with negative curvature. We focus on the Siegel half-space $\HH$ for clarity. While the method can be generalized, it is best understood through this example.

The space $\mathcal H$ consists of $n\times n$ complex symmetric matrices whose imaginary part is positive definite, written 
\begin{equation}
\label{eq:siegel1}  
\mathcal H= \big\{ Z=X+\ii \,Y\big| X=X^T, \;Y=Y^T, \; Y\succ 0\big\}.
\end{equation}
It is equipped with a metric $g_\HH$ expressed in terms of the 1-forms $dZ$ as follows
\begin{equation}
\label{eq:siegel2}
    g_\HH=\Tr(Y^{-1}\ud Z\,Y^{-1}\ud \bar Z). 
\end{equation}

This space was first studied in depth by Siegel~\cite{Siegel}. He showed that $\mathcal{H}$ is a symmetric space with negative curvature that generalizes the \Poincare\/ upper half-plane~\cite[Thms 1-4]{Siegel}.  \Mobius\/ transformations of the \Poincare\/ upper half-plane, $w= (az+b)/(cz+d)$, where $\left( \begin{array}{ll} a & b \\ c & d \end{array}\right) \in \mathrm{SL}(2,\R)$,  are replaced by automorphisms of $\HH$ of the form
\begin{equation}
\label{eq:siegel-auto}
W = (AZ + B) (CZ+D)^{-1}.
\end{equation}
Here the $n\times n$ real matrices $A,B,C,D$ correspond to elements of the symplectic group $\mathrm{Sp}(n,\mathbb R)$ as follows. They are the elements of block matrices $M$ such that
\begin{equation}
M = \left( \begin{array}{ll} A & B \\ C & D \end{array}\right), \quad M^TJM =J, \quad J = \left( \begin{array}{rr} 0 & I \\ -I & 0 \end{array}\right),
\end{equation}
where $I$ is the $n\times n$ identity matrix.

Siegel characterized automorphisms of $\HH$ that map a given pair of points $(Z,Z_1)$ into a pair $(W,W_1)$ by introducing the matrix-valued cross-ratio 
\begin{align}
\label{eq:siegel3}
    \mathfrak R(Z,Z_1)=(Z-Z_1)(Z-\bar Z_1)^{-1}(\bar Z-\bar Z_1)(\bar Z-Z_1)^{-1}.
\end{align}
The invariants of the group action~\eqref{eq:siegel-auto} are the eigenvalues of $\mathfrak R(Z,Z_1)$: there is an automorphism of $\HH$ mapping $(Z,Z_1)$ into $(W,W_1)$ if and only if $\mathfrak R(Z,Z_1)$ and $\mathfrak R(W,W_1)$ have the same eigenvalues~\cite[Thm. 2]{Siegel}. These eigenvalues always lie within the interval $[0,1]$ (a proof is presented in Section~\ref{sec:th1-proof}).

We construct a stochastic process $Z_t \in \HH$ such that if 
$\lambda_t = (\lambda^1_t,\ldots, \lambda_t^n)$ are the eigenvalues of $\mathfrak{R}(Z_t,iI)$, and $\sigma^k_t>0$ is the positive solution to the implicit equation 
\begin{equation}
\label{eq:siegel-eig1}
    \tanh^2 \frac{\sigma^k_t}{2} = \lambda_t^k, \quad 1\leq k \leq n,
\end{equation}
then $\sigma_t = (\sigma^1_t,\ldots, \sigma_t^n)$ has the same law as the unique strong solution to the \Ito\/ differential equation
\begin{equation}
\label{eq:siegel-eig2} 
\ud \sigma^k_t=\frac12(\coth{\sigma^k_t}+\sum_{l\neq k}\frac{\sinh{\sigma^k_t}}{\cosh{\sigma^k_t}-\cosh{\sigma^l_t}})\ud t+\sqrt{\frac2\beta}\ud W^k_t, \quad k=1,\ldots, n.
\end{equation}
Here $W_t=(W^1_t,\ldots,W^n_t)$ is a standard Wiener process in $\R^n$. 

Equation~\eqref{eq:siegel-eig2} is an analogue of Dyson Brownian motion because they both correspond to stochastic gradient ascent in a  geometric sense. For these reasons, we call the interacting particle system  in equation~\eqref{eq:siegel-eig2} {\em Siegel Brownian motion\/}. 

Let us now introduce geometric stochastic gradient ascent through examples. We then state two theorems that describe this structure for  Siegel Brownian motion. The underlying well-posedness theory parallels that of Dyson Brownian motion. We show in Section~\ref{sec:wellposedness} that equation~\eqref{eq:siegel-eig2} has global strong solutions for $\beta \geq 2$.

\subsection{Motion by curvature and stochastic flows}
Let $z\in \R^n$, let $r=|z|$ and let $P_{z}$ and $P_{z}^\perp$ denote the orthonormal projections onto the tangent and normal space to the circle $S^{n-1}_r$ of radius $r$ at the point $z$. Let $B_t$ be a standard Wiener processes in $\mathbb R^n$ and consider the \Ito\/ SDE 
\begin{align}\label{eq:toy2}
\ud z_t = P_{z_t} \, \ud  B_t + \sqrt{\frac2\beta} P_{z_t}^\perp \,\ud B_t.
\end{align}
This SDE is designed so that the point $z_t$ evolves because of independent tangential and normal stochastic fluctuations with the parameter $\beta>0$ describing the anisotropy of these fluctuations. In the limit $\beta=\infty$, there are only tangential fluctuations and we use \Ito's formula to observe that 
\begin{equation}
\label{eq:toy3}
\ud(|z_t|^2) = \ud(z_t^Tz_t) = \ud z_t^T z_t + z_t^T \ud z_t + \ud z_t^T\ud z_t = (n-1) \, \ud t. 
\end{equation}
It follows that the radius $r_t$ evolves deterministically by 
\begin{equation}
\label{eq:toy3a}
\dot{r}_t= \frac{n-1}{2r_t}.
\end{equation}
Thus, when $\beta=\infty$ each point $z_t$ evolves tangentially and stochastically, but the spheres $S^{n-1}_{r_t}$ evolve deterministically by motion by (minus a half times) curvature.

For arbitrary $\beta \in (0,\infty]$, a similar calculation yields that the radius $r_t$ has the same law as the solution to the SDE
\begin{equation}
\label{eq:toy3b}
    \ud r_t=\frac{n-1}{2r_t} \ud t + \sqrt{\frac2\beta}\ud W_t
    =\frac12\nabla S(r_t)\ud t + \sqrt{\frac2\beta}\ud W_t.
\end{equation}
Here $W_t$ is a standard one-dimensional Wiener process,  and 
\[ S(r)=\log \mathrm{vol} \,(S^{n-1}_{r}),\] 
with $\mathrm{vol}(S^{n-1}(r))$ denoting the $n-1$ dimensional volume of the sphere $S^{n-1}_{r}$, is the Boltzmann entropy. Observe  again that the diffusion of $r_t$ is due to the normal fluctuations driven by $B_t$, whereas the tangential fluctuations in $z_t$ give rise to the {\em normal\/} drift through the \Ito\/ correction.

\subsection{General principles}
This example can be systematized through the use of Riemannian submersion.  Euclidean space $\R^n$ is foliated by concentric circles of radius $r>0$ centered at the origin. The space of radii, $r \in [0,\infty)$, may be identified with the quotient space $\R^n/\mathrm O(n)$ and equipped with a metric via Riemannian submersion. When constructing the SDE~\eqref{eq:toy2}, we introduced stochastic forcing upstairs, with independent horizontal and vertical fluctuations and an anisotropy parameter $\beta>0$.

In a recent paper~\cite{HIM22}, this approach was used to construct Dyson Brownian motion at inverse temperature $\beta \in (0,\infty]$. In this case, the space upstairs is the space of Hermitian matrices equipped with the Frobenius norm, writen $\left(\mathbb{H}(n),\mathrm{Fr}\right)$; the space downstairs is the Weyl chamber of eigenvalues; and the foliation corresponds to isospectral matrices generated by the unitary group~\cite[\S 2.1]{HIM22}. Each matrix $M \in \mathbb{H}(n)$ may be diagonalized, written $M=U\Lambda U^*$, so that it lies on a $U(n)$ orbit $\mathcal{O}_\Lambda \subset \mathrm{Her}(n)$. We again use the projections $P_M$ and $P_M^\perp$ onto the tangent and normal spaces to $T_M\mathcal{O}_\Lambda$ to define the \Ito\/ differential equation
\begin{equation}
\label{eq:ito-rmt}
dM_t = P_{M_t} dX_t + \sqrt{\frac{2}{\beta}} P_{M_t}^\perp dX_t,
\end{equation}
where $X_t$ is the standard Wiener process in $\left(\mathbb{H}(n),\mathrm{Fr}\right)$.

Equation~\eqref{eq:ito-rmt} is a model for Dyson Brownian motion in the following sense. It is shown in~\cite[Thm. 1]{HIM22} that the eigenvalues of $M_t$, written $\lambda_t=(\lambda^1_t,\ldots, \lambda_t^n)$, have the same law as the unique strong solution to the \Ito\/ differential equation
\begin{equation}
\label{eq:ito-rmt2}
d\lambda^k_t = \sum_{j \neq k} \frac{1}{\lambda_k-\lambda_j} \, dt + \sqrt{\frac{2}{\beta}} dW^k_t, \quad 1\leq k \leq n.
\end{equation}
Here $W_t$ denotes the standard Wiener process in $\R^n$ and we assume, given $M_0 \in \mathbb{H}(n)$, that $\lambda_0\in \R^n$ is the spectrum of $M_0$ presented in increasing order.  

Equation~\eqref{eq:ito-rmt2} may also be written as 
\begin{equation}
\label{eq:ito-rmt2c}
d\lambda_t = \frac{1}{2}\nabla S(\lambda) \, dt + \sqrt{\frac{2}{\beta}} dW_t, \quad S(\lambda) = \log \mathrm{vol} (\mathcal{O}_{\lambda}).
\end{equation}
The gradient in $\lambda$ is with respect to the usual Euclidean metric. This is the metric on $\R^n$ obtained by pushing forward the Frobenius metric on $\mathbb{H}(n)$ under the map that sends a matrix to its eigenvalues. This is the natural Riemannian submersion in this problem. Again, in the limit $\beta=\infty$, the group orbits $\mathcal{O}_{\lambda_t}$ evolve deterministically by motion by minus a half times mean curvature.

\subsection{A new matrix model}
Dyson Brownian motion is the fundamental interacting particle system in random matrix theory. Despite its extensive study, the geometric properties revealed by the construction in~\cite{HIM22} have not been noticed before. The simplicity of equation~\eqref{eq:ito-rmt} suggests that other interesting particle systems may be generated through this method. This viewpoint motivated the construction in this paper. 

The use of the Siegel half-space allows us to illustrate our method in a symmetric space with negative curvature. In order to do so, we must replace each term in equation~\eqref{eq:ito-rmt} by a geometric analogue. 

First, the driving process $X_t$ in equation~\eqref{eq:ito-rmt} is replaced by Browian motion, denoted $\bd B_t$,  in $(\mathcal H,g)$. This process may be constructed in several ways, as explained in Section~\ref{sec:background} below; we follow~\cite{Hsu,Ikeda}.  

The analogue of the foliation of $\R^n$ by concentric circles must be provided by a foliation of $\HH$ using a suitable group of isometries (observe that an $O(n)$ action generates the concentric circles). We fix the point $iI_n$ as the origin in $\HH$ so that its stabilizer subgroup $G_{\ii I_n}<\mathrm{Sp}(n,\mathbb R)$ provides the desired group of isometries. Let $\pi:\mathcal H\to \mathcal H/G_{\ii I_n}$ be the canonical projection from $\mathcal H$ to the quotient space. By Siegel's theorem, the moduli space under this group action may be parametrized by the eigenvalues of the cross-ratio $\mathfrak R(Z,\ii I_n)$. Precisely, the quotient space is
\begin{align}
    \mathcal H/G_{\ii I_n}=\{ \lambda \in [0,1)^n | \lambda_1 \leq \lambda_2 \ldots \leq \lambda_n \}\,.
\end{align}
To make the quotient space a manifold it is necessary to avoid singularities at the boundary, and we must ensure the motion remains in
\begin{align}\label{eq:chamber}
    \mathcal Q=\mathrm{Int}(\mathcal H/G_{\ii I_n})=\{ \lambda \in (0,1)^n | \lambda_1 < \lambda_2 \ldots < \lambda_n \}\,.
\end{align}
This space is also equivalent to the quotient of a subset $\tilde{\mathcal H}\subset \mathcal H$ by $G_{\ii I_n}$
\begin{align}
\nonumber
    \tilde{\mathcal H}=&\{\,Z\in\mathcal H\,|\,\mathfrak R(Z,\ii I_n)\text{ has no zero or repeated eigenvalues }\,\}, \;\;
    \mathcal Q=&\tilde{\mathcal H}/G_{\ii I_n}\,.
\end{align}
It is clear that $\mathcal Q$ is a manifold. We equip it with a metric using Riemannian submersion as described below. 


For each $\lambda\in\mathcal Q$, define the subset of $\HH$
\begin{equation}
\label{eq:group-orbit1}
\mathcal O_\lambda= \{Z \in \HH \left|\; \mathrm{eig}\,\mathfrak R(Z,\ii I_n)=\lambda\right. \}.
\end{equation}

Then $\mathcal O_\lambda$ is a smooth manifold immersed in $\mathcal H$ and it is the $G_{\ii I_n}$ group orbit of any $Z$ in it. For each $Z \in \mathcal{O}_\lambda$ we then define the orthogonal projection operators $P_Z,P_Z^{\perp}$
to the tangent space $T_Z\mathcal O_\lambda$ and its orthogonal complement $T_Z\mathcal O_\lambda^{\perp}$ respectively.


Fix $\beta \in (0,\infty]$ and assume given $\bd Z_0$ such that the eigenvalues of $\mathfrak{R}(\bd Z_0, i\,I)$ are distinct and lie in $(0,1)$.  We consider the (formal) \Ito\/ stochastic differential equation 
\begin{align}\label{eq:sdeZ}
    \ud \bd Z_t=P_{Z_t}(\ud \bd B_t)+\sqrt{\frac2\beta}P^{\perp}_{Z_t}(\ud \bd B_t), \quad \bd Z(0) = \bd{Z}_0.
\end{align}
This expression is only formal because stochastic differential equations in $\HH$ must be defined using the Stratonovich rule to ensure coordinate invariance. The precise definition of equation~\eqref{eq:sdeZ}, as well as its well-posedness theory, is presented in Section~\ref{sec3} and Section~\ref{sec:wellposedness}. We define the stopping time $\tau_\beta$ as the first time at which eigenvalues to $\mathfrak R(Z_t,\ii I)$ collide. In particular, $\tau_\beta=+\infty$ for $\beta\geq 2$. We then have the following
\begin{theorem}\label{thm1}
    Let $\bd Z_t$ be the unique strong solution to the SDE~\eqref{eq:sdeZ}, let $\lambda_t = \mathrm{eig}\left(\mathfrak R(\bd Z_t,\ii I_n\right))$ and define $\sigma_t$ using equation~\eqref{eq:siegel-eig1}. Then $\sigma_t$ has the same law as the unique strong solution to the SDE~\qref{eq:siegel-eig2} on the maximal interval of existence $[0,\tau_\beta)$.
\end{theorem}
Define the Boltzmann entropy of the group orbit $\mathcal O_{\lambda}$ by the formula 
   \begin{align}
   \label{eq:entropy}
        S(\lambda) = \log\mathrm{vol}(\mathcal O_{\lambda}),
    \end{align}
where the volume form is obtained from the metric $g$ on $\HH$. We changed variables from $\lambda$ to $\sigma$ in equation~\eqref{eq:siegel-eig1} in order that the metric in $\sigma$-space is the standard Euclidean metric on $\R^n$. Thus, $\nabla S(\sigma)$ in equation~\eqref{eq:evol-sigma} below denotes the usual Euclidean gradient of the entropy $S(\sigma)$. 

\begin{theorem}\label{thm2} There is a universal constant $c_n>0$ such that 
\begin{equation}
    \label{eq:entropy-formula}
    S(\sigma) =\sum_{k=1}^n\log\sinh{\sigma^k}+\sum_{1\leq k<l\leq n}\log|\cosh{\sigma^k}-\cosh{\sigma^l}|+c_n.
\end{equation}
Equation \eqref{eq:siegel-eig2} is equivalent to stochastic gradient ascent 
of Boltzmann entropy 
    \begin{align}
    \label{eq:evol-sigma}
        \ud \sigma_t= \frac{1}{2}\nabla S(\sigma_t)\ud t +\sqrt{\frac2\beta}\ud W_t.
    \end{align}    
\end{theorem}

\begin{remark}
We have adopted the Boltzmann sign convention for entropy: it is chosen so that our systems evolve to maximize entropy.  Conceptually this corresponds to maximizing disorder in the gauge group in each of our examples. 
\end{remark}

\begin{remark}
We use the convention in geometric analysis to define the mean curvature of an immersion $u: \mathcal{N} \to (\mathcal{M},g)$ of a smooth manifold $\mathcal{N}$ into a Riemannian manifold $(\mathcal{M},g)$): it is the trace of the second fundamental form of $u$ with respect to the metric $g$. As a guide, it is helpful to note that the mean curvature of a sphere $S_r^{n-1}$ points inwards and it has magnitude $(n-1)/r$.    
\end{remark}

\begin{remark}
The mean curvature is usually obtained as the first variation of a volume functional. However, for Riemannian submersion of group orbits generated by isometries, the mean curvature is also the negative of the gradient of the entropy (evaluated upstairs)~\cite[p.3350]{Pacini}. This is why in all the examples above the group orbits move by motion by minus a half curvature in the limit $\beta=\infty$.    
\end{remark}

\begin{remark}
It is well known that Dyson Brownian motion may be generalized using representation theory to the Dunkl process and radial Heckman-Opdam process (see~\cite[\S 2]{HM} for a comprehensive treatment of these processes, including a history and explicit examples). It is natural to expect a relationship between these processes and Siegel Brownian motion (compare for example equation~\eqref{eq:siegel-eig2} and~\cite[Eq. 2.35]{HM}). However, we have been unable to determine such a link. Our approach makes no use of representation theory beyond what is implicit in Siegel's classical work on $\HH$; instead our tools are a combination of standard stochastic differential geometry and facts from~\cite{Siegel}.
\end{remark}


The following section provides background material about SDE on manifolds. This is followed by the proofs of Theorem~\ref{thm1} and Theorem~\ref{thm2}. The well-posedness theory is described in Section~\ref{sec:wellposedness}.

\section{Background: SDE and martingales on manifolds}
\label{sec:background}
\subsection{Outline}
In this section, we assume given a smooth manifold $\mathcal{M}$ (a differentiable manifold with a $C^\infty$ structure) and we review some facts about SDE and martingales on manifolds. The main objective is to explain the definition of equation~\eqref{eq:sdeZ}. Our primary sources for stochastic differential geometry are~\cite{Hsu,Ikeda}.

\begin{notation} We use bold font to denote $\mathcal{M}$-valued processes and vector fields on $\mathcal{M}$.  Coordinates of these processes are written in the standard font.
For example, let $\varphi=(\varphi^1,...,\varphi^d)$ be local coordinates on an open set $U\subset\mathcal M$. The $i$-th coordinate of an $\mathcal M$-valued process $\bd X_t$ is denoted by $X^i_t=\varphi^i(\bd X_t)$. A vector field $\bd V$ on $\mathcal M$ can be described in coordinates as $\bd V(X)=V^i(X)\partial_i=V^i(X)\frac{\partial}{\partial \varphi^i}$, $X\in U$. For a function $f\in C^{\infty}(\mathcal M)$, $\bd V(f)=\bd V(\ud f)=\ud f(\bd V)$ gives the directional derivative of $f$ along $\bd V$.
\end{notation}

\subsection{SDE, martingales on manifolds}\label{sec3}
The Stratonovich formulation is used to define SDE on manifolds. A typical SDE on $\mathcal{M}$ is as follows. Assume given $l$ vector fields  $\{\bd V_1,...,\bd V_l\}$ on $\mathcal M$, an $\mathbb R^l$-valued driving semimartingale $Z$, and write 
   \begin{align}
   \label{eq:sde-manifold1}
        \ud \bd X_t=\sum_{\alpha=1}^l\bd V_{\alpha}(\bd X_t)\circ\ud Z^{\alpha}_t.
    \end{align}
The use of the Stratonovich formulation ensures that the usual chain rule applies, ensuring coordinate invariance of the solution to the SDE. The solution theory uses semimartingales in the following way.  

\begin{definition}
    Let $\mathcal M$ be a smooth manifold and $(\Omega,\mathcal F_*,\mathbb P)$ a filtered probability space. Let $\tau$ be an $\mathcal F_*$-stopping time. A continuous $\mathcal M$-valued process $\bd X_t$ defined on $[0,\tau)$ is called an $\mathcal M$-valued semimartingale if $f(\bd X_t)$ is a real valued semimartingale on $[0,\tau)$ for all $f\in C^{\infty}(\mathcal M)$.
\end{definition}
\begin{definition}
    An $\mathcal M$-valued semimartingale $\bd X$ defined up to a stopping time $\tau$ is a solution to the SDE ~\eqref{eq:sde-manifold1} if for all $f\in C^{\infty}(\mathcal M)$, 
    \begin{align}\label{sde1}
        f(\bd X_t)=f(\bd X_0)+\sum_{\alpha=1}^l\int_0^t\bd V_{\alpha}f(\bd X_t)\circ \ud Z^{\alpha}_s, \quad0\leq t<\tau.
    \end{align}
\end{definition}
These two definitions are intrinsic. Only the differentiable structure on $\mathcal{M}$ has been used to define the SDE~\eqref{eq:sde-manifold1}, since we have only used the concept of vector fields. However, in order to define martingales taking values in $\mathcal{M}$, the manifold must be equipped with a connection $\nabla$. For Riemannian manifolds, we always use the Levi-Civita connection.
\begin{definition}
    Assume $\mathcal M$ is a smooth manifold equipped with a connection $\nabla$. An $\mathcal M$-valued semimartingale $\bd X$ is called a $\nabla$-martingale if and only if the following process is a $\mathbb R$-valued local martingale for all $f\in C^{\infty}(\mathcal M)$:
    \begin{align}\label{mtg1}
        N^f(\bd X)_t:=f(\bd X_t)-f(\bd X_0)-\frac12\int_0^t\nabla^2f(\ud \bd X_s,\ud \bd X_s),
    \end{align}
    where $\nabla^2f$ is the Hessian of $f$ under $\nabla$.
\end{definition}

In our model, the driving semimartingales are Wiener processes or just $t$. The SDE we are concerned with have the following form:
\begin{align}\label{sde2}
    \ud \bd X_t=\sum_{\alpha=1}^l\Big(\bd V_{\alpha}(\bd X_t)\circ\ud W^{\alpha}_t-\frac12\nabla_{\bd V_{\alpha}}\bd V_{\alpha}(\bd X_t)\ud t\Big)
\end{align}
where $(W^1_t,...,W^l_t)$ are independent Wiener processes, and $\nabla_{\bd V_{\alpha}}\bd V_{\alpha}$ is the covariant derivative of $\bd V_{\alpha}$ along itself. The deterministic correction is introduced in equation~\eqref{sde2} in order that we could formally regard equation~\eqref{sde2} as the \Ito\/ SDE
\begin{equation}
\label{sde3}
\ud \bd X_t = \sum_{\alpha=1}^l\bd V_{\alpha}(\bd X_t)\, \ud W^{\alpha}_t.    
\end{equation} 
More precisely, we have 
\begin{proposition}\label{prop}
    The solution to the SDE~\eqref{sde2} is a $\nabla$-martingale.
\end{proposition}
\begin{proof}
    Inserting equation \eqref{sde1} into equation \eqref{mtg1}, we only need to show that the following process is a local martingale~\cite[Prop 2.5.2,pp. 56]{Hsu} :
    \begin{align}
        N^f(\bd X)_t=\sum_{\alpha=1}^l\Big(\int_0^t\Big(\bd V_{\alpha}f(\bd X_t)\circ \ud W^{\alpha}_s-\frac12(\nabla_{\bd V_{\alpha}}\bd V_{\alpha})f(\bd X_t)\Big)\ud t-\frac12\int_0^t\nabla^2f(\bd V_{\alpha},\bd V_{\alpha})(\bd X_t)\ud t\Big).
    \end{align}

    But since $\bd V(\bd Vf)=(\nabla_{\bd V}\bd V)f+\nabla_{\bd V}(\ud f)(\bd V)=(\nabla_{\bd V}\bd V)f+\nabla^2f(\bd V,\bd V)$, we combine the last two terms to obtain 
    \begin{align}
        N^f(\bd X)_t&=\sum_{\alpha=1}^l\Big(\int_0^t\bd V_{\alpha}f(\bd X_t)\circ\ud W^{\alpha}_s-\frac12\int_0^t\bd V_{\alpha}(\bd V_{\alpha}f)(\bd X_t)\ud t\Big)\nonumber\\
        &=\sum_{\alpha=1}^l\int_0^t\bd V_{\alpha}f(\bd X_t)\ud W^{\alpha}_t.
    \end{align}
\end{proof}

\subsection{Brownian motion and orthogonal projections}
When the manifold is Riemannian we may characterize Brownian motion in several ways. Orthonormal bases provide a direct approach.
\begin{corollary}\label{corBM}
Let $\{\bd V_\alpha\}_{\alpha=1}^n$ be an orthonormal basis for $T\mathcal M$. Then the solution to the SDE 
\begin{align}\label{sdeBM}
    \ud \bd X_t=\sum_{\alpha=1}^n(\bd V_{\alpha}\circ\ud W^{\alpha}_t-\frac12\nabla_{\bd V_{\alpha}}\bd V_{\alpha}\ud t)
\end{align}
has generator $\Delta=\Tr(\nabla^2)$, so that it is a Brownian motion on $\mathcal M$.
\end{corollary}

\begin{proof}
    For any $f\in C^\infty(\mathcal M)$, we have 
    \begin{align}
        f(\bd X_t)-f(\bd X_0)=&\sum_{\alpha=1}^n\int_0^t\Big(\bd V_\alpha f(\bd X_s)\circ\ud W^\alpha_s-\frac12 \nabla_{\bd V_\alpha}\bd V_\alpha f(\bd X_s)\ud s\Big)\\
        =&\sum_{\alpha=1}^n\int_0^t\frac12\nabla^2f(\bd V_\alpha,\bd V_\alpha)(\bd X_s)\ud s\nonumber\\
        =&\int_0^t\frac12\Delta f(\bd X_s)\ud s\nonumber\,.
    \end{align}
\end{proof}

At this stage we are able to define the projection operator $P$ onto a submanifold $\mathcal O\subset\mathcal M$ clearly. For drift terms in an SDE, no new ideas are required: we simply project the drift vector to $T\mathcal O$. For diffusion term that is an $\nabla$-martingale in an SDE, e.g. with a single driving vector field,
\begin{align}\label{mtgSde}
    \ud \bd X_t=\bd V\circ\ud W_t-\frac12\nabla_{\bd V}\bd V\ud t,
\end{align}
we define 
\begin{align}\label{proj-sde}
    P(\ud \bd X_t)=P(\bd V)\circ \ud W_t-\frac12\nabla_{P(\bd V)}P(\bd V)\ud t\, . 
\end{align}

To understand this projection, we have the following Lemma:
\begin{lemma}
    Let $f\in C^\infty(\mathcal M)$, $\bd X_t$ solves SDE \eqref{mtgSde} and is a $\nabla$-martingale. By \eqref{mtg1}, we could write
    \begin{align}
        \ud f(\bd X_t)=\frac12\nabla^2f(\bd V,\bd V)\ud t+\text{ some local martingale}\,.
    \end{align}

    Let $\bd X^p_t$ be the solution to the SDE:
    \begin{align}
        \ud \bd X^p_t=P(\bd V)\circ \ud W_t-\frac12\nabla_{P(\bd V}P(\bd V)\ud t\, . 
    \end{align}
    then we have
    \begin{align}
        \ud f(\bd X^p_t)=\frac12\nabla^2f(P(\bd V),P(\bd V))\ud t+\text{ some local martingale}\,.
    \end{align}
\end{lemma}

For an SDE with both diffusion and drift terms, we decompose the RHS of the SDE into the sum of a $\nabla$-martingale diffusion term and drift term and project them respectively. The result is that the evolution of $f(X^p_t)$, for $f\in C^\infty(\mathcal M)$ and a solution to the projected SDE $\bd X^p_t$, only contains the values of $\nabla f$ and $\nabla^2 f$ in the subspace to be projected.

When we have an orthonormal basis $\{\bd V_{\alpha}\}_{\alpha=1}^n$ of $T\mathcal M$ such that $\{\bd V_{\alpha}\}_{\alpha=1}^d$ is also an orthonormal basis of $T\mathcal O$, and $\bd X_t$ satisfies equation \eqref{sdeBM}, we have 
\begin{align}
\label{eq:proj1}
    P(\ud \bd X_t)=&\sum_{\alpha=1}^d(\bd V_{\alpha}\circ\ud W^{\alpha}_t-\frac12\nabla_{\bd V_{\alpha}}\bd V_{\alpha}\ud t),\\
    \label{eq:proj2}
    P^{\perp}(\ud \bd X_t)=&\sum_{\alpha=d+1}^d(\bd V_{\alpha}\circ\ud W^{\alpha}_t-\frac12\nabla_{\bd V_{\alpha}}\bd V_{\alpha}\ud t).
\end{align}
This is the form in which we will prove Theorem~\ref{thm1} and Theorem~\ref{thm2}.

\subsection{Some matrix algebra to be used}\label{sec:matrix}
In this section, we introduce some matrix algebra to be frequently used in the following sections.

We first introduce a matrix representation of $\mathfrak u(n)$, the Lie algebra of unitary group $\mathrm U(n)$. Define
\begin{align}
    &\alpha_{k,l}:=\frac{e_le_k^T-e_ke_l^T}{\sqrt 2}\,,&&1\leq k<l\leq n\,,\\
    &\beta_{k}:=\ii e_ke_k\,,&&1\leq k\leq n\,,\\
    &\beta_{k,l}:=\ii\frac{e_le_k^T+e_ke_l^T}{\sqrt 2}\,,&&1\leq k<l\leq n\,.
\end{align}
then $\mathfrak u(n)$ is the linear space spanned by these vectors
\begin{align}\label{un_matrix}
    &\mathfrak u(n)=\mathrm{Span}(\{\alpha_{k,l}\}_{1\leq k<l\leq n}^n\cup\{\beta_{k}\}_{1\leq k\leq n}\cup\{\beta_{k,l}\}_{1\leq k<l\leq n})\,,\\
    &\qquad=\Big\{A\in M_{n\times n}|A=-A^\dag:=-(\bar A)^T\Big\}\,.\nonumber
\end{align}
The validity of this statement can be verified by dimension counting.

We should note that this basis of $\mathfrak u(n)$ is also orthonormal under the inner product $\mathrm{Re}(\mathrm{Tr}(UV^\dag))$:
\begin{align}\label{matrix_orth}
    &\mathrm{Re}(\mathrm{Tr}(\alpha_{i,j}\alpha_{k,l}^\dag))=\delta_{ik}\delta_{jl}\,,
    &&\mathrm{Re}(\mathrm{Tr}(\beta_i\beta_k^\dag))=\delta_{ik}\,,
    &&\mathrm{Re}(\mathrm{Tr}(\beta_{i,j}\beta_{k,l}^\dag))=\delta_{ik}\delta_{jl}\,,\\
    &\mathrm{Re}(\mathrm{Tr}(\alpha_{i,j}\beta_{k}^\dag))=0\,,
    &&\mathrm{Re}(\mathrm{Tr}(\alpha_{i,j}\beta_{k,l}^\dag))=0\,,
    &&\mathrm{Re}(\mathrm{Tr}(\beta_{i}\beta_{k,l}^\dag))=0\,.\nonumber
\end{align}

Their commutator is defined to be $[A,B]:=AB-BA$. We are not going to write down all commutators, but the following property will be used:
\begin{align}\label{commutator_perp}
    \mathrm{Tr}([A,B]B^\dag)=-\mathrm{Tr}([A,B]B)=-\mathrm{Tr}(AB^2-BAB)=0\,.
\end{align}
which means that the commutator $[A,B]$ is perpendicular to $B$ (and $A$ by symmetry) under inner product. By the completeness of the inner product $\mathrm{Re}(\mathrm{Tr}(UV^\dag))$, we have that the commutator $[A,B]$, for $A,B$ chosen from the basis above, is a linear combination of basis vectors other than $A$ and $B$.

In the analysis of symmetric matrix $X\in\mathrm{Sym}(n)$, its diagonalization $X=PM P^T$ always exists, where $P\in\mathrm O(n)$ and $M$ is $n\times n$ diagonal matrix. For complex symmetric matrices we also have the following decomposition.
\begin{lemma}[Takagi, \protect{\cite[Theorem 2]{Takagi}}]
\label{lem:takagi}
    Each complex symmetric matrix $A$ can be decomposed to $A=QM Q^T$, where $M=\mathrm{diag}\,(\mu_1,...,\mu_n)$ is a real nonnegative diagonal matrix and $Q \in \mathrm U(n)$.
\end{lemma}

Such decomposition provides an isomorphism between linear vector spaces: the tangent spaces $T\mathrm{Sym}(n,\mathbb C)$ and $T\mathbb R^{n}\otimes \mathfrak u(n)$. To be specific, we consider the following basis of $T\mathrm{Sym}(n,\mathbb C)$: define $E^R$, $E^I$ matrices as:
\begin{align}\label{symn_matrix}
    &E^R_{k}=e_ke_k^T\,,&&E^I_k=\ii e_ke_k^T\,,&&1\leq k\leq n\,,\\
    &E^I_{k,l}=\ii\frac{e_ke_l^T+e_le_k^T}{\sqrt 2}\,,&&E^R_{k,l}=\frac{e_ke_l^T+e_le_k^T}{\sqrt 2}\,,&&1\leq k<l\leq n\,.
\end{align}
and define tangent vector fields:
\begin{align}\label{tangent_vectors1}
    &\bd l_k:&&\bd l_k(\ud X)=QE^R_kQ^T\,,&&1\leq k\leq n\,,\\
    &\bd u_{(k)}:&&\bd u_{(k)}(\ud X)=2\mu_kQE^I_kQ^T\,,&&1\leq k\leq n\,,\\
    &\bd u_{(k,l,1)}:&&\bd u_{(k,l,1)}(\ud X)=(\mu_k+\mu_l)QE^I_{k,l}Q^T\,,&&1\leq k<l\leq n\,,\\
    &\bd u_{(k,l,2)}:&&\bd u_{(k,l,2)}(\ud X)=(\mu_k-\mu_l)QE^R_{k,l}Q^T\,,&&1\leq k<l\leq n\,.
\end{align}
We also use the following basis of $T\mathbb R^n\otimes \mathfrak u(n)$:
\begin{align}\label{tangent_vectors2}
    &\bd l'_k:&&\bd l'_k(\ud M)=e_ke_k^T\,,&&\bd l'_k(\ud Q)=0\,,&&1\leq k\leq n\,,\\
    &\bd u'_{(k)}:&&\bd u'_{(k)}(\ud M)=0\,,&&\bd u'_{(k)}(\ud Q)=Q\beta_k\,,&&1\leq k\leq n\,,\\
    &\bd u'_{(k,l,1)}:&&\bd u'_{(k,l,1)}(\ud M)=0\,,&&\bd u'_{(k,l,1)}(\ud Q)=Q\beta_{k,l}\,,&&1\leq k<l\leq n\,,\\
    &\bd u'_{(k,l,2)}:&&\bd u'_{(k,l,2)}(\ud M)=0\,,&&\bd u'_{(k,l,2)}(\ud Q)=Q\alpha_{k,l}\,,&&1\leq k<l\leq n\,.
\end{align}

We claim that these two set of vector fields are exactly the same vector fields. More precisely, let $\pi:\mathbb R^n\otimes \mathrm U(n)\to \mathrm{Sym}(n,\mathbb C),\pi(M,Q)=QM Q^T$ be the bijection between $\mathbb R^n\otimes \mathrm U(n)$ and $\mathrm{Sym}(n,\mathbb C)$, and $\pi_*:T\mathbb R^n\otimes \mathfrak u(n)\to T\mathrm{Sym}(n,\mathbb C)$ be the differential of $\pi$, or the pushforward map, we have $\pi_*(\bd v')=\bd v$ for $\bd v$ being any vector fields defined above.

To see this, one should check the following matrix identities:
\begin{align}
    &[e_ke_k^T,M]=0\,,&&[\alpha_{k,l},M]=(\mu_k-\mu_l)E^R_{k,l}\,,\nonumber\\
    &\{\beta_k,M\}=2\mu_kE^I_k\,,&&\{\beta_{k,l},M\}=(\mu_k+\mu_l)E^I_{k,l}\,.\nonumber 
\end{align}
for $M$ being diagonal matrix, where $\{A,B\}=AB+BA$ is the anti-commutator. Then by taking the differential of $X=QM Q^T$:
\begin{align}\label{dR_dLdQ}
    \ud X=&Q(\ud M+Q^{-1}\ud QM+M \ud Q^T(Q^T)^{-1})Q^T\\
(\ud\omega:=Q^{-1}\ud Q)=&Q(\ud M+ \ud\omega M+M\ud\omega^T)Q^T\nonumber\\
=&Q(\ud M+\ud\omega M- M\bar{\ud \omega})Q^T\nonumber\\
=&Q(\ud M+[\mathrm{Re}(\ud\omega), M]+\{\ii\,\mathrm{Im}(\ud\omega), M\})Q^T\,,\nonumber
\end{align}
we may verify for example that
\begin{align}
    \bd u'_{(k,l,1)}(\ud X)=&Q\{\beta_{k,l}, M\}Q^T=Q(\mu_k+\mu_l)(\frac{e_ke_l^T+e_le_k^T}{\sqrt 2})Q^T\\
    =&(\mu_k+\mu_l)QE^I_{k,l}Q^T=\bd u_{(k,l,1)}(\ud X)\,.\nonumber
\end{align}

\section{Proof of Theorem~\ref{thm1}}
\label{sec:th1-proof}
The proof of Theorem~\ref{thm1} consists of two parts.  First, we express the SDE~\eqref{eq:sdeZ} for $\bd Z_t$ on $\mathcal H$ in a more explicit form using the orthonormal bases described in Section~\ref{sec:background}. Next, we use eigenvalue perturbation theory, the definition of the cross-ratio $\mathfrak{R}(Z_t,iI_n)$ and the change of variables from $\lambda_t$ to $\sigma_t$ to show that $\sigma_t$ satisfies \eqref{eq:siegel-eig2}. Almost all the work lies in explicit computations related to the change of basis. The eigenvalue perturbation theory is standard, except near repeated eigenvalues. The well-posedness theory in Section~\ref{sec:wellposedness} allows us to focus on the case of distinct eigenvalues.

The outline of the computation of the basis is as follows. Following Siegel, we first change variables from $\HH$ to the equivalent of the unit disk for complex symmetric matrices (the domain $\DD$ below). This is a convenient coordinate system, since it makes several properties of the cross-ratio matrix obvious.  We will find an orthonormal basis $\{\bd L_k\}_{k=1}^n\cup \{\bd U_i\}_{i=1}^{n^2}$ for $T\mathcal H$, such that $\{\bd U_i\}_{i=1}^{n^2}$ is also an orthonormal basis of $T\mathcal O_{\lambda}$, and $\{\bd L_k\}_{k=1}^n\subset T\mathcal O_{\lambda}^{\perp}$. By Riemannian submersion, $\{\bd L_k\}_{k=1}$ also gives an orthonormal basis of $T\Lambda$ (identifying $\bd L_k$ with $\ud\pi(\bd L_k)$). Then by Corollary \ref{corBM}, we are able to construct Brownian motion on $(\mathcal H,g)$; equations \eqref{eq:proj1} and~\eqref{eq:proj2} then complete the definition of the SDE~\eqref{eq:siegel-eig2}.

\subsection{The unit disk $\DD$}
The domain of symmetric complex matrices defined by
\begin{equation}
    \label{eq:defD}
    \DD = \{ W =W^T \left| I - W\bar{W} \succ 0 \right. \}
\end{equation}
is the analogue of the open unit disk in $\C$. Since $W=W^T$, the matrix $I- W\bar{W}$ is Hermitian and the notation $I-W\bar{W} \succ 0$ means that it is positive definite. 

The conformal mapping $z \mapsto \frac{z-i}{z+i}$ maps the \Poincare\/ upper half-plane to the unit disk in $\C$. We write the analogous transformation for the Siegel half-space $\HH$ as
\begin{equation}
\label{eq:siegel5}
R=\varphi(Z)=(Z-\ii I_n)(Z+\ii I_n)^{-1}.    
\end{equation}
Since $Z=Z^T$ and these matrices commute, we have  $R=R^T$ for $Z \in \HH$, a property that we will use repeatedly. Siegel showed that $\DD=\varphi(\HH)$ by adapting the  proof of conformal equivalence between the upper half-plane and the unit disk to several complex variables~(see~\cite[II.4]{Siegel}). 

The transformation~\eqref{eq:siegel5} plays an important role in our work because the cross-ratio matrix $\mathfrak{R}(Z,\ii I_n)$ may be written
\begin{equation}
    \label{eq:crossratiodecomp}
    \mathfrak{R}(Z,\ii I_n) = R \bar{R}.
\end{equation}
Since $R=R^T$ we find immediately that 
\begin{equation}
    \label{eq:crossratiodecomp2}
    \bar{\mathfrak{R}}^T = (\bar{R}R)^T=R^T\bar{R}^T = R\bar{R}= \mathfrak{R}.
\end{equation}
Thus, $\mathfrak{R}$ is Hermitian positive semi-definite. This property allows us to obtain precise descriptions of the orbits $\mathcal O_{\lambda}$ and the metric $g_\HH$ as outlined below. We begin with a classical 
\begin{lemma}
\label{lem:DD-char}
The domain $\DD$ may be expressed as
\begin{align}
\label{eq:DD-explicit}
   \DD = \varphi(\mathcal H)=\big\{Q M Q^T\big|Q\in U(n),  M\text{ is diagonal},\;0\leq \mu_i<1,i=1,...,n\big\}.
\end{align}
\end{lemma}
\begin{proof}
The first equality in equation~\eqref{eq:DD-explicit} follows from~\cite[II.4]{Siegel}. The second equality  is also straightforward. We decompose $R=Q M Q^T$ using Lemma~\ref{lem:takagi} and invert equation~\eqref{eq:siegel5} to find
\begin{align}
    \mathrm{Im}(Z)=&(I_n-R)^{-1}(I_n-R\bar R)(I_n-\bar R)^{-1}\nonumber\\
    =&(\bar{Q}^T- M Q^T)^{-1}(I_n- M^2)(Q-\bar Q M)^{-1}.
\end{align}
It is immediate that $\mathrm{Im}(Z)\succ 0$ if and only if $\mu_i^2<1$ for $i=1, \ldots, n$. 
\end{proof}
\begin{lemma}
    For $Z\in\mathcal H$ and $\lambda=\mathrm{eig}\;\mathfrak R(Z,\ii I_n)$, the orbit $\mathcal O_{\lambda}$ is represented by 
\begin{align}
    \mathcal O_{\lambda}=\varphi^{-1}\Big(\{Q M Q^T| M=\mathrm{diag}(\mu_1,...,\mu_n),Q\in U(n)\}\Big)
\end{align}
where $\lambda=(\mu_1^2,...,\mu_n^2)$.
\end{lemma}
\begin{proof}
    By equation~\eqref{eq:crossratiodecomp} and Lemma~\ref{lem:takagi}, the eigenvalues of $\mathfrak R(Z,\ii I_n)$ are just the diagonal elements $\{\mu_i^2\}_{i=1}^n$. Thus, 
    \begin{align}
        \varphi(\mathcal O_{\lambda})=\{Q M Q^T| M=\mathrm{diag}(\mu_1,...,\mu_n),Q\in U(n)\}
    \end{align}
    where $\lambda=(\mu_1^2,...,\mu_n^2)$.
\end{proof}

\subsection{The orthonormal basis}
Now we will construct an orthonormal basis of $T\mathcal H$ and $T\mathcal O_{\lambda}$. As $\varphi:\mathcal H\to\mathcal D$ is an analytic bijection and we will equip $\mathcal D$ with metric that is compatible with $\varphi$ later, we can regard $(\mathcal H,g_{\mathcal H})$ and $(\mathcal D,g_{\mathcal D})$ as the same Riemannian manifold, while $Z$ and $R$ are just two coordinate systems on the this manifold. Also, vector fields related by the pushforward map $\ud\pi$ (e.g. $\bd u$ in \eqref{tangent_vectors1} and $\bd u'$ in \eqref{tangent_vectors2}) are regarded as representations of the same vector field under different coordinate systems.

In this section we will mainly use $R$ and its diagonalization $R=Q M Q^T$ as coordinate system. According to section~\ref{sec:matrix}, we already have a basis \eqref{tangent_vectors1} of $T_R\mathcal D=T_R\mathrm{Sym}(n,\mathbb C)$ for $R\in\mathcal D$.

By definition of $\bd u$ vector fields \eqref{tangent_vectors2}, we also know that for each $Z\in\mathcal O_{\lambda}$
\begin{align}
    T_Z\mathcal O_{\lambda}=\mathrm{Span}\big(\{\bd u_{(k)}\}_{k=1}^n\cup\{\bd u_{(k,l,1)}\}_{1\leq k<l\leq n}\cup\{\bd u_{(k,l,2)}\}_{1\leq k<l\leq n}\big).
\end{align}

Now we will compute the metric $g_{\mathcal D}$ that is compatible with settings above. In view of bilinear form, we just need to compute the transformation coefficient of $\ud\pi$ and get $g_{\mathcal D}$. In a geometric point of view, we have $g_{\mathcal D}=g_{\mathcal H}$ in the sense that they are the same tensor in $\mathfrak X(T^*\mathcal H\otimes T^*\mathcal H)$ represented by different coordinate systems. Under the geometric point of view, we use 
\begin{align}
    R=&(Z-\ii I_n)(Z+\ii I_n)^{-1},\quad Z=\ii(I_n+R)(I_n-R)^{-1}\,,\\
    \ud Z=&2\ii (I_n-R)^{-1}\ud R(I_n-R)^{-1}\,,\nonumber\\
\text{ and }    Y=&\frac1{2\ii}(Z-\bar Z)=(I_n-R)^{-1}(I_n-R\bar R)(I_n-\bar R)^{-1}\\
=&\bar Y=(I_n-\bar R)^{-1}(I_n-\bar R R)(I_n-R)^{-1}\,,\nonumber
\end{align}
and compute
\begin{align}
    g_{\mathcal H}=&\mathrm{Tr}(Y^{-1}\ud Z Y^{-1}\ud \bar Z)=\mathrm{Tr}(Y^{-1}\ud Z (\bar Y)^{-1}\ud \bar Z)\\
    =&\mathrm{Tr}\Big([(I_n-\bar R)(I_n-R\bar R)^{-1}(I_n-R)][2\ii (I_n-R)^{-1}\ud R(I_n-R)^{-1}]\nonumber\\
    &\quad [(I_n-R)(I_n-\bar R R)^{-1}(I_n-\bar R)][(-2\ii)(I_n-\bar R)^{-1}\ud \bar R(I_n-\bar R)^{-1}]\Big)\nonumber\\
    =&4\,\mathrm{Tr}\Big((I_n-R\bar R)^{-1}\ud R(I_n-\bar RR)^{-1}\ud \bar R\Big)=g_{\mathcal D}. \nonumber
\end{align}

Notice that when $n=1$, $g_{\mathcal D}=\frac{4\ud z\ud\bar z}{(1-|z|^2)^2}$ coincides with the \Poincare\/ metric on the unit disk. From now on we will only use $g_{\mathcal H}$ for both $g_{\mathcal H}$ and $g_{\mathcal D}$, as they are conceptually the same geometric object. In the following theorem, we will show that the metric $g_{\mathcal H}$ is diagonalized under the basis \eqref{tangent_vectors1}. Thus, if we normalize the basis \eqref{tangent_vectors1} we will get an orthonormal basis.


\begin{theorem}\label{uvectors}
    The following vector fields $\{\bd L_k\}_{k=1}^n\cup\{\bd U_{(k)}\}_{k=1}^n\cup \{\bd U_{(k,l,m)}\}_{1\leq k<l\leq n,m=1,2}$ give an orthonormal basis of $\mathfrak X(T\mathcal H)$. Moreover, $\{\bd U_{(k)}(Z)\}_{k=1}^n\cup \{\bd U_{(k,l,m)}(Z)\}_{1\leq k<l\leq n,m=1,2}$ give an orthonormal basis of $T_Z\mathcal O_\lambda$ for each $Z\in\mathcal O_\lambda$ and $\lambda\in\mathcal W$.
    
    At a point $\varphi(Z)=R=Q M Q^T$, where $ M=\mathrm{diag}(\tanh\frac{\sigma^1}2,...,\tanh\frac{\sigma^n}2)$, the vector fields are defined as follows
    \begin{align}\label{orthonormal_basis}
        &\bd L_k(\ud R)=\frac1{2\cosh^2{\frac{\sigma^k}2}}Q\Big(e_ke_k^T\Big)Q^T,&&1\leq k\leq n,\\
        &\bd U_{(k)}(\ud R)=\frac1{2\cosh^2{\frac{\sigma^k}2}}Q\Big(\ii e_ke_k^T\Big)Q^T,&&1\leq k\leq n,\\
        &\bd U_{(k,l,1)}(\ud R)=\frac1{2\cosh{\frac{\sigma^k}2}\cosh{\frac{\sigma^l}2}}Q\Big(\frac{\ii(e_ke_l^T+e_le_k^T)}{\sqrt 2}\Big)Q^T,&&1\leq k<l\leq n,\\
        &\bd U_{(k,l,2)}(\ud R)=\frac1{2\cosh{\frac{\sigma^k}2}\cosh{\frac{\sigma^l}2}}Q\Big(\frac{e_ke_l^T+e_le_k^T}{\sqrt 2}\Big)Q^T,&&1\leq k<l\leq n.
    \end{align}
\end{theorem}
In order to simplify notation, we will use $\{\bd U_i\}_{i=1}^{n^2}$ to denote the union $\{\bd U_{(k)}\}_{k=1}^n\cup \{\bd U_{(k,l,m)}\}_{1\leq k<l\leq n,m=1,2}$ and use $\bd U_i$ for any vector field in $\{\bd U_i\}_{i=1}^{n^2}$, when there is no confusion.

\begin{proof}
    Notice that all $\bd L$ and $\bd U$ vectors are just normalized $\bd l$ and $\bd u$ vectors. We only need to prove that $g_{\mathcal H}$ is diagonalized under $\bd l$ and $\bd u$ basis \eqref{tangent_vectors1}, and extract all coefficient of squares.

    To do this, we first transform $g_{\mathcal H}$ into $( M,Q)$ system. Use \eqref{dR_dLdQ}, we get $g_{\mathcal H}$
    \begin{align}
        g_{\mathcal H}=&4\,\mathrm{Tr}\Big((I_n-R\bar R)^{-1}\ud R(I_n-\bar RR)^{-1}\ud \bar R\Big)\\
        =&4\,\mathrm{Tr}\Big((I_n- M^2)^{-1}(\ud M+[\mathrm{Re}(\ud\omega), M]+\{\ii\,\mathrm{Im}(\ud\omega), M\})\nonumber\\
        &\quad (I_n- M^2)^{-1}(\ud M+[\mathrm{Re}(\ud\omega), M]-\{\ii\,\mathrm{Im}(\ud\omega), M\})\Big)\nonumber\\
        =&4\Big(\mathrm{Tr}((I_n- M^2)^{-1}(\ud M+[\mathrm{Re}(\ud\omega), M])(I_n- M^2)^{-1}(\ud M+[\mathrm{Re}(\ud\omega), M]))\nonumber\\
        &\quad +\mathrm{Tr}((I_n- M^2)^{-1}\{\mathrm{Im}(\ud\omega), M\}(I_n- M^2)^{-1}\{\mathrm{Im}(\ud\omega), M\})\Big)\nonumber\\
        =&4\Big(\mathrm{Tr}((I_n- M^2)^{-2}\ud  M\ud M)\nonumber\\
        &\quad +\mathrm{Tr}((I_n- M^2)^{-1}[\mathrm{Re}(\ud\omega), M](I_n- M^2)^{-1}[\mathrm{Re}(\ud\omega), M])\nonumber\\
        &\quad +\mathrm{Tr}((I_n- M^2)^{-1}\{\mathrm{Im}(\ud\omega), M\}(I_n- M^2)^{-1}\{\mathrm{Im}(\ud\omega), M\})\Big)\,.\nonumber
    \end{align}
    where we used the fact that $(I_n- M^2)$ and $\ud  M$ are diagonal and $[\mathrm{Re}(\ud\omega), M]$ has zero diagonal.

    Take linear combination of basis with constant coefficients
    \begin{align}
        \bd v=\sum_{k=1}^nc_{k}\bd l_{k}+\sum_{k=1}^nd_k\bd u_{(k)}+\sum_{1\leq k< l\leq n}d_{k,l,1}\bd u_{(k,l,1)}+\sum_{1\leq k< l\leq n}d_{k,l,2}\bd u_{(k,l,2)},
    \end{align}
    we find that $g_{\mathcal H}(\bd v,\bd v)$ becomes a sum of squares
    \begin{align}\label{eq:sum_square}
        g_{\mathcal H}(\bd v,\bd v)=&\sum_{k=1}^n \frac{4}{(1-\mu_k^2)^2}c_{k}^2+\sum_{k=1}^n\frac{4(2\mu_k)^2}{(1-\mu_k^2)^2}d_k^2\\
        &+\sum_{1\leq k< l\leq n}\frac{4(\mu_k+\mu_l)^2}{(1-\mu_k^2)(1-\mu_l^2)}d_{k,l,1}^2+\sum_{1\leq k< l\leq n}\frac{4(\mu_k-\mu_l)^2}{(1-\mu_k^2)(1-\mu_l^2)}d_{k,l,2}^2.\nonumber
    \end{align}
    Thus, the vectors $\{\bd l_{k}\}_{k=1}^n\cup\{\bd u_{(k)}\}_{k=1}^n\cup\{\bd u_{(k,l,1)}\}_{1\leq k< l\leq n}\cup\{\bd u_{(k,l,2)}\}_{1\leq k< l\leq n}$ give an orthogonal basis. We normalize this basis to obtain an orthonormal basis \eqref{orthonormal_basis}.

    For example, we can check that 
    \begin{align}
        g_{\mathcal H}(\bd L_k,\bd L_k)=&g_{\mathcal H}(\frac1{2\cosh\frac{\sigma^k}2}\bd l_k,\frac1{2\cosh\frac{\sigma^k}2}\bd l_k)\nonumber\\
        =&\frac{4}{(1-\mu_k^2)^2}(\frac1{2\cosh\frac{\sigma^k}2})^2=1\,,\nonumber\\
        g_{\mathcal H}(\bd U_{(k,l,2)},\bd U_{(k,l,2)})=&g_{\mathcal H}(\frac1{2\cosh{\frac{\sigma^k}2}\cosh{\frac{\sigma^l}2}}\frac{\bd u_{(k,l,2)}}{(\mu_k-\mu_l)},\frac1{2\cosh{\frac{\sigma^k}2}\cosh{\frac{\sigma^l}2}}\frac{\bd u_{(k,l,2)}}{(\mu_k-\mu_l)})\nonumber\\
        =&(\frac1{2\cosh{\frac{\sigma^k}2}\cosh{\frac{\sigma^l}2}(\mu_k-\mu_l)})^2\frac{4(\mu_k-\mu_l)^2}{(1-\mu_k^2)(1-\mu_l^2)}=1\,.\nonumber
    \end{align}
\end{proof}

\subsection{Covariant derivatives}
To determine all terms in SDE \ref{sdeBM}, we have the following Lemma for the expression of $\nabla_{\bd V}\bd V$ for all $\bd L$ and $\bd U$ vectors.
\begin{lemma}\label{lemCoD}
    The covariant derivatives $\nabla_{\bd U}\bd U$ and $\nabla_{\bd L}\bd L$ are given by:
    \begin{align}\label{covariant_derivative}
        \nabla_{\bd L_k}\bd L_k=&0\,,&&1\leq k\leq n\,,\\
        \nabla_{\bd U_{(k)}}\bd U_{(k)}=&-\coth{\sigma^k}\bd L_k\,,&&1\leq k\leq n\,,\\
        \nabla_{\bd U_{(k,l,1)}}\bd U_{(k,l,1)}=&-\coth(\frac{\sigma^k}2+\frac{\sigma^l}2)\frac{\bd L_k+\bd L_l}2\,,&&1\leq k<l\leq n\,,\\
        \nabla_{\bd U_{(k,l,2)}}\bd U_{(k,l,2)}=&-\coth(\frac{\sigma^k}2-\frac{\sigma^l}2)\frac{\bd L_k-\bd L_l}2\,,&&1\leq k<l\leq n\,.
    \end{align}
\end{lemma}

\begin{proof}
    For a vector field $\bd V\in\mathfrak X(T\mathcal H)$, We could compute $\nabla_{\bd V}\bd V$ by computing $\langle \nabla_{\bd V}\bd V,\bd V'\rangle$ using Koszul's formula \cite[Theorem 3.6, Page 55]{do1992riemannian}, for $\bd V'=\sum_{k=1}^n a_k\bd L_k+\sum_{i=1}^{n^2}b_i \bd U_i$ with arbitrary constant coefficients $a_k$'s and $b_i$'s, where $\langle\cdot,\cdot\rangle=g_{\mathcal H}(\cdot,\cdot)$.

    Koszul's formula is simplified in our case because all inner products between $\bd U$ and $\bd L$ vectors are constants. For example,
    \begin{align}
        \langle \nabla_{\bd U_i}\bd U_i,\bd L_k\rangle=&\bd U_i\langle\bd U_i,\bd L_k\rangle-\frac12\bd L_k\langle \bd U_i,\bd U_i\rangle-\langle[\bd U_i,\bd L_k],\bd U_i\rangle=\langle [\bd L_k,\bd U_i],\bd U_i\rangle\,,
    \end{align}
    where the first two terms vanish because $\langle \bd U_i,\bd U_i\rangle=1$ and $\langle \bd U_i,\bd L_k\rangle=0$.

    We give a proof of the expression of $\nabla_{\bd L_k}\bd L_k$ and $\nabla_{\bd U_{(k,l,2)}}\bd U_{(k,l,2)}$, then the computations are similar for others.
    
    To prove $\nabla_{\bd L_k}\bd L_k=0$, we need a fact discovered by Siegel in \cite[Theorem 3, Page 3]{Siegel}, implying that the translation of each $\sigma^k$ is a geodesic. Also, notice that 
    \begin{align}\label{eq:L_orthonormal}
        \bd L_k(\sigma^l)=&(2\cosh^2\frac{\sigma^l}2)\bd L_k(\tanh \frac{\sigma^l}2)
        =(2\cosh^2\frac{\sigma^l}2)\bd L_k(\mu_l)=\delta_{kl}\,,
    \end{align}
    which means $\bd L_k=\big(\frac{\partial}{\partial \sigma^k}\big)$, say $\bd L_k$ is the unit vector field tangent to the geodesic of $\sigma^k$ translation. Therefore $\nabla_{\bd L_k}\bd L_k=0$.
 
    For $\bd U_{(k,l,2)}$, according to the formula above, we first need to distinguish those vector fields that are commutative with $\bd U_{(k,l,2)}$, or whose commutator with $\bd U_{(k,l,2)}$ are perpendicular to $\bd U_{(k,l,2)}$.

    We first argue that $\langle \nabla_{\bd U_i}\bd U_i,\bd U_j\rangle=0$ for all $\bd U_i,\bd U_j$ vectors. This is a result from \eqref{commutator_perp}: the commutator $[\bd U_i,\bd U_j]$ does not contain $\bd U_i$ or $\bd U_j$ term.

    For $\langle \nabla_{\bd U_i}\bd U_i,\bd L_m\rangle$, $\bd u_{(k,l,2)}$ is by definition \eqref{tangent_vectors2} commutative with $\bd l_m$ (thus $\bd L_m$), so 
    \begin{align}
        [\bd L_m,\bd U_{(k,l,2)}]=&[(\frac{\partial}{\partial\sigma^m}),\frac1{2\sinh(\frac{\sigma^k-\sigma^l}2)}\bd u_{(k,l,2)}]\\
        =&(\frac{\partial}{\partial\sigma^m}(\frac{1}{2\sinh(\frac{\sigma^k-\sigma^l}2)}))\bd u_{(k,l,2)}\nonumber\\
        =&-\coth(\frac{\sigma^k-\sigma^l}2)(\frac{\delta_{km}-\delta_{lm}}2)\bd U_{(k,l,2)}\,.\nonumber
    \end{align}

    Thus, $\nabla_{\bd U_{(k,l,2)}}\bd U_{(k,l,2)}$ is given by orthonormal decomposition below:
    \begin{align}
    \nabla_{\bd U_{(k,l,2)}}\bd U_{(k,l,2)}=&\sum_{m=1}^n
        \langle \nabla_{\bd U_{(k,l,2)}}\bd U_{(k,l,2)},\bd L_m\rangle\bd L_m\\
        =&\sum_{m=1}^n\langle [\bd L_m,\bd U_{(k,l,2)}],\bd U_{(k,l,2)}\rangle\bd L_m\nonumber\\
        =&\sum_{m=1}^n-\coth(\frac{\sigma^k-\sigma^l}2)(\frac{\delta_{km}-\delta_{lm}}2)\bd L_m\nonumber\\
        =&-\coth(\frac{\sigma^k-\sigma^l}2)(\frac{\bd L_k-\bd L_l}2)\,.\nonumber
    \end{align}
\end{proof}

\subsection{Evolution of spectrum}\label{sec43}
Now it suffice to prove theorem \ref{thm1}, i.e. to compute the evolution of spectrum $\lambda(\bd Z_t)=(\tanh^2{\frac{\sigma^1}2},...,\tanh^2{\frac{\sigma^n}2})$.

\begin{proof}[Proof of theorem 1]
    Let $ M_t= M(\bd Z_t)$. By definition \eqref{sde1}, the evolution of $ M_t$ satisfies the following SDE:
    \begin{align}
        \ud  M_t=&\sum_{i=1}^{n^2}\Big(\bd U_i( M_t)\circ\ud \tilde W^i_t-\frac12\nabla_{\bd U_i}\bd U_i( M_t)\ud t\Big)\\
        &\quad+\sqrt{\frac2\beta}\sum_{k=1}^{n}\Big(\bd L_k( M_t)\circ\ud W^k_t-\frac12\nabla_{\bd L_k}\bd L_k( M_t)\ud t\Big)\nonumber\\
        =&\sum_{i=1}^{n^2}-\frac12\nabla_{\bd U_i}\bd U_i( M_t)\ud t+\sqrt{\frac2\beta}\sum_{k=1}^n\bd L_k( M_t)\circ\ud W^k_t\,,\nonumber
    \end{align}
where $\{\tilde W^i_t\}_{i=1}^{n^2}$ and $\{W^k_t\}_{k=1}^n$ are $n^2+n$ independent Wiener processes.

Two facts were used in the computation above: $\bd U_i(\ud  M)=0$ for all $\bd U_i\in \{\bd U_i\}_{i=1}^{n^2}$ by definition \eqref{tangent_vectors2}, and $\nabla_{\bd L_k}\bd L_k=0$ for $k=1,...,n$ from Lemma \ref{lemCoD}.


Using the expression of $\nabla_{\bd U_i}\bd U_i$ in \eqref{covariant_derivative}, and $\bd L_k(\mu_l)=\frac1{2\cosh^2{\frac{\sigma^k}2}}\delta_{kl}$ which is equivalent to $\bd L_k( M_t)=\frac{e_ke_k^T}{2\cosh^2{\frac{\sigma^k}2}}$, we can write down the SDE exactly as:
\begin{align}
    \ud  M_t=&\sum_{k=1}^n \ud\big(\tanh(\frac{\sigma^k}2)\big)e_ke_k^T\\
    =&\sum_{k=1}^n\frac12\coth(\sigma^k_t)\frac{e_ke_k^T}{2\cosh^2{\frac{\sigma^k}2}}\ud t
    +\sum_{1\leq k<l\leq n}\frac12\coth(\frac{\sigma^k_t+\sigma^l_t}2)\frac{\frac{e_ke_k^T}{2\cosh^2{\frac{\sigma^k}2}}+\frac{e_le_l^T}{2\cosh^2{\frac{\sigma^l}2}}}2\ud t\nonumber\\
    &+\sum_{1\leq k<l\leq n}\frac12\coth(\frac{\sigma^k_t-\sigma^l_t}2)\frac{\frac{e_ke_k^T}{2\cosh^2{\frac{\sigma^k}2}}-\frac{e_le_l^T}{2\cosh^2{\frac{\sigma^l}2}}}2\ud t+\sqrt{\frac2\beta}\sum_{k=1}^n \frac{e_ke_k^T}{2\cosh^2{\frac{\sigma^k}2}}\circ\ud W^k_t\nonumber\\
    =&\sum_{k=1}^n\Big(\frac1{2\cosh^2(\frac{\sigma^k}2)}\times\frac12\big(\coth(\frac{\sigma^k_t}2)+\sum_{l\neq k}\frac{\sinh(\sigma^k_t)}{\cosh(\sigma^k_t)-\cosh(\sigma^l_t)}\big)\ud t\nonumber\\
    &\quad+\sqrt{\frac2\beta}\frac1{2\cosh^2(\frac{\sigma^k}2)}\circ\ud W^k_t\Big)e_ke_k^T\nonumber
\end{align}

As $\{e_ke_k^T\}_{k=1}^n$ are linearly independent, the coefficients on the two sides coincide, and we get the evolution of $\tanh(\frac{\sigma^k_t}2)$. Using chain rule we find the evolution of $\sigma^k_t$:
\begin{align}
    \ud\sigma^k_t=&\frac12\big(\coth(\frac{\sigma^k_t}2)+\sum_{l\neq k}\frac{\sinh(\sigma^k_t)}{\cosh(\sigma^k_t)-\cosh(\sigma^l_t)}\big)\ud t+\sqrt{\frac2\beta}\times 1\circ\ud W^k_t\nonumber\\
    =&\frac12\big(\coth(\frac{\sigma^k_t}2)+\sum_{l\neq k}\frac{\sinh(\sigma^k_t)}{\cosh(\sigma^k_t)-\cosh(\sigma^l_t)}\big)\ud t+\sqrt{\frac2\beta}\ud W^k_t\nonumber
\end{align}
which is \eqref{eq:siegel-eig2}.
\end{proof}



\section{Proof of Theorem 2}
\label{sec:thm2-proof}
In this section, we evaluate the Boltzmann entropy $S: \mathcal{Q} \to \R$ defined in equation~\eqref{eq:entropy} and prove Theorem~\ref{thm2}.


\subsection{Orbit volume}
For each $\lambda\in\mathcal{Q}$, the orbit $\mathcal O_{\lambda}$ is naturally a Riemannian submanifold of $\mathcal H$ equipped with the induced metric $g_{\mathcal O_\lambda}$, defined by
\begin{align}
    g_{\mathcal O_\lambda}(\bd v,\bd v')=g_\mathcal H(\bd v,\bd v')\,,
\end{align}
for $\bd v,\bd v'\in T\mathcal O_{\lambda}$.

We first introduce a collection of $n^2$ $1$-forms to be the basis of $\mathfrak X(T^*\mathcal O_{\lambda})$. Denoted $\{\omega^i\}_{i=1}^{n^2}\subset\mathfrak X (T^*\mathcal O_{\lambda})$ that are dual to the orthogonal basis $\{\bd u^i\}_{i=1}^{n^2}$ defined in \eqref{tangent_vectors1}. That is, 
\begin{align}
    \bd u^i(\omega^j)=\delta_{ij}\,,&&i,j=1,...,n^2\,,
\end{align}
using \eqref{eq:sum_square} the metric $g_{\mathcal O_\lambda}$ has the form
\begin{align}
    g_{\mathcal O_\lambda}=&\sum_{k=1}^n\frac{4(2\mu_k)^2}{(1-\mu_k^2)^2}\omega^{(k)}\otimes\omega^{(k)}\\
    &+\sum_{1\leq k< l\leq n}\frac{4(\mu_k+\mu_l)^2}{(1-\mu_k^2)(1-\mu_l^2)}\omega^{(k,l,1)}\otimes\omega^{(k,l,1)}\nonumber\\
    &+\sum_{1\leq k< l\leq n}\frac{4(\mu_k-\mu_l)^2}{(1-\mu_k^2)(1-\mu_l^2)}\omega^{(k,l,2)}\otimes\omega^{(k,l,2)}\nonumber\\
    =&4\Big(\sum_{k=1}^n\sinh^2\sigma^k\omega^{(k)}\otimes\omega^{(k)}+\sum_{1\leq k< l\leq n}\sinh^2(\frac{\sigma^k+\sigma^l}2)\omega^{(k,l,1)}\otimes\omega^{(k,l,1)}\nonumber\\
    &+\sum_{1\leq k< l\leq n}\sinh^2(\frac{\sigma^k-\sigma^l}2)\omega^{(k,l,2)}\otimes\omega^{(k,l,2)}\Big)\,.\nonumber
\end{align}

Let $M_{\lambda,\omega}$ be the matrix of $g_{\mathcal O_{\lambda}}$ under basis $\{\omega^i\}_{i=1}^{n^2}$, then the volume form is $\ud V=\sqrt{\det(M_{\lambda,\omega})}\bigwedge_{i=1}^{n^2}\omega^i$ and $\mathrm{vol}(\mathcal O_{\lambda})=\int_{\mathcal O_{\lambda}}\ud V$ gives the volume of orbit $\mathcal O_\lambda$.

Notice that $\det(M_{\lambda,\omega})$ only depends on $\lambda$, and is a constant on $\mathcal O_\lambda$, we can extract it out and get 
\begin{align}
    \mathrm{vol}(\mathcal O_\lambda)=2^{n^2}\prod_{k=1}^n\sinh\sigma^k\prod_{1\leq k<l\leq n}\sinh(\frac{\sigma^k+\sigma^l}2)\sinh(\frac{\sigma^k-\sigma^l}2)\int_{\mathcal O_\lambda}\bigwedge_{i=1}^{n^2}\omega^i
\end{align}
and we will argue that the integral $\int_{\mathcal O_\lambda}\bigwedge_{i=1}^{n^2}\omega^i$ is a constant that does not depend on $\lambda$, so we have the following 
\begin{theorem}
    For $\lambda\in \mathcal Q$, the volume of orbit $\mathcal O_\lambda$ is given by 
    \begin{align}
        \mathrm{vol}(\mathcal O_\lambda)=C_n\prod_{k=1}^n\sinh\sigma^k\prod_{1\leq k<l\leq n}\sinh(\frac{\sigma^k+\sigma^l}2)\sinh(\frac{\sigma^k-\sigma^l}2)
    \end{align}
    where $C_n$ is a constant that only depends on $n$.
\end{theorem}

\begin{proof}
    We only need to show that the integral $\int_{\mathcal O_\lambda}\bigwedge_{i=1}^{n^2}\omega^i=\int_{\mathcal O_\lambda'}\bigwedge_{i=1}^{n^2}\omega^i$ for $\lambda,\lambda'\in\mathcal Q$. To do this we consider a bijection
    \begin{equation}
    \label{eq:vol1}
    \phi:\mathcal O_{\lambda}\to\mathcal O_{\lambda'}\nonumber\,,\\
    \text{ for }Z=\varphi^{-1}(Q M Q^T)\in\mathcal O_{\lambda}\,,\, \phi(Z)=\varphi^{-1}(Q M' Q^T)\,.
    \end{equation}
    i.e. $\phi$ is the map that preserves the $\mathrm U(n)$ component of $R$ coordinate.
    
    Since $\phi$ is a diffeomorphism between $\mathcal O_\lambda$ and $\mathcal O_{\lambda'}$, using $\phi^*(\omega^i\circ \phi)\in\mathfrak X(T^*\mathcal O_{\lambda})$, the pullback of $\omega^i\circ\phi:\mathcal O_{\lambda}\to T^*\mathcal O_{\lambda'}$, we can use change of variable formula to get
    \begin{align}
       \int_{\mathcal O_{\lambda'}}\bigwedge_{i=1}^{n^2}\omega^i=&\int_{\phi(\mathcal O_{\lambda})}\bigwedge_{i=1}^{n^2}\omega^i\\
        =&\int_{\mathcal O_{\lambda}}\bigwedge_{i=1}^{n^2}\phi^*(\omega^i\circ \phi)\,,\nonumber
    \end{align}

    Now we will show that $\phi^*(\omega^i\circ \phi)=\omega^i$. This is the direct result of $\phi_*(\bd u_i)=\bd u_i\circ\phi$, to prove which consider $Z'=\phi(Z)$ for $Z\in\mathcal O_\lambda$ and an arbitrary function $f\in C^\infty(\mathcal O_{\lambda'})$, then we need to prove that $\phi_*(\bd u_i(Z))f=\bd u_i(Z')f$.

    Let the matrix-valued derivative $\frac{\partial f}{\partial Q}$ defined by $\ud f=\mathrm{Tr}(\frac{\partial f}{\partial Q}\ud Q)$, we have 
    \begin{align}
        \phi_*(\bd u_i(Z))f=&\bd u_i(f\circ\phi(Z))=\mathrm{Tr}\Big((\frac{\partial f}{\partial Q})_{Z'}(\bd u_i(\ud Q))_Z\Big)\,,\nonumber\\
        \bd u_i(Z')f=&\mathrm{Tr}\Big((\frac{\partial f}{\partial Q})_{Z'}(\bd u_i(\ud Q))_{Z'}\Big)\nonumber\,,
    \end{align}
    where we used the fact that $Z$ and $Z'$ has the same $Q$ component. As $(\bd u_i(\ud Q))_{Z'}=(\bd u_i(\ud Q))_Z$ from \eqref{tangent_vectors2}, we proved that $\phi_*(\bd u_i(Z))f=\bd u_i(Z')f$.
    
    Then by 
    \begin{align}
        \bd u_j(\phi^*(\omega^i\circ \phi))=\phi_*(\bd u_j)(\omega^i\circ\phi)=(\bd u_i\circ\phi)(\omega^i\circ\phi)=\delta^i_j\,,
    \end{align}
    we proved that $\phi^*(\omega^i\circ \phi)=\omega^i$.

    So the integral is a constant:
    \begin{align}
        \int_{\mathcal O_{\lambda'}}\bigwedge_{i=1}^{n^2}\omega^i=\int_{\mathcal O_{\lambda}}\bigwedge_{i=1}^{n^2}\phi^*(\omega^i\circ \phi)=\int_{\mathcal O_{\lambda}}\bigwedge_{i=1}^{n^2}\omega^i\,.
    \end{align}
\end{proof}

Take logarithm and using the formula $\sinh(\frac{a+b}2)\sinh(\frac{a-b}2)=\frac{\cosh(a)-\cosh(b)}2$, we get the Boltzmann entropy
\begin{align}
    S(\sigma)=\log\mathrm{vol}(\mathcal O_\lambda)=\sum_{k=1}^n\log\sinh\sigma^k+\sum_{1\leq k<l\leq n}\log|\cosh\sigma^k-\cosh\sigma^l|+c_n\,.
\end{align}

\subsection{Metric and gradient flow}\label{sec52}
We have already computed the Boltzmann entropy function $S(\sigma)$, and to find a gradient flow we will also need a metric $g_{\Lambda}$ in the spectrum space $\mathcal Q$. The natural selection is to make the projection $\pi:\mathcal H\to\mathcal Q$ a Riemannian submersion.

Now as we already have an orthonormal basis of $T\mathcal H$ and it is clear that
\begin{align}
    \mathrm{Ker}(\ud\pi)=\mathrm{Span}\Big(\{\bd U_i\}_{i=1}^{n^2}\Big)\,,\qquad \mathrm{Ker}(\ud\pi)^{\perp}=\mathrm{Span}\Big(\{\bd L_k\}_{k=1}^n\Big)\,.
\end{align}

By definition of a Riemannian submersion, $\{\ud\pi(\bd L_k)\}_{k=1}^n$ will be an orthonormal basis in $T\mathcal Q$. Using \eqref{eq:L_orthonormal} we have 
\begin{align}
    g_{\mathcal Q}=\sum_{k=1}^n \ud\sigma^k\ud\sigma^k\,,\qquad
    (g_{\mathcal Q})^{-1}=\sum_{k=1}^n \bd L_k\otimes \bd L_k\,.
\end{align}

Therefore, the gradient of $S(\sigma)$ is just the usual differential $\mathrm{grad}\,S(\ud\sigma^k)=\frac{\partial S}{\partial \sigma^k}$. Comparing SDE \eqref{eq:siegel-eig2} and the entropy \eqref{eq:entropy-formula} we conclude that the evolution of $(\sigma^1,...,\sigma^n)$ is just the Langevin equation maximizing Boltzmann entropy $S$.

\section{Well-posedness and Non-colliding property}
\label{sec:wellposedness}
In this section, we prove the existence and uniqueness of strong solution of SDE \eqref{eq:siegel-eig2}, when $\beta\geq 2$ and the initial condition satisfies $\lambda_0\in\mathcal Q$, i.e. $\sigma_0=(\sigma^i_0)_{i=1}^n$ are all distinct and non-zero. The essential observation is that Boltzmann entropy $S(\sigma_t)$ is a local submartingale and with probability 1 it is bounded below, which avoids collision of eigenvalues. This strategy is commonly used for the well-posedness of Dyson Brownian motion in random matrix theory, and our proof adapts arguments used in~\cite{AGZ,BDS}.

\begin{theorem}\label{thm:non-collision}
    When $\beta\geq 2$, with initial condition $\sigma_0=(\sigma^i_0)_{i=1}^n$ such that $S(\sigma_0)>-\infty$, there exists a unique strong solution to the SDE \eqref{eq:siegel-eig2}, and $S(\sigma_t)>-\infty$ for all $t>0$ almost surely.
\end{theorem}

Our overall strategy is to consider an auxiliary SDE with a cut-off of the coefficients in \eqref{eq:siegel-eig2} such that well-posedness is guaranteed, and properly define some stopping times so that the solution to the auxiliary SDE also solves SDE~\eqref{eq:siegel-eig2} before they stopped. We then define the solution of SDE \eqref{eq:siegel-eig2} to be equal to the solution of auxiliary SDE before stopping, then we let the stopping times goes to infinity to obtain a solution almost surely defined for all $t>0$.

The auxiliary SDE is as follows: define $N(\sigma)=\log(\sum_{i=1}^n\cosh\sigma^i)$, for $k>-S(\sigma_0),K>N(\sigma_0)$, let $\sigma^{(k,K)}_t$ solves the SDE as follows:
    \begin{align}\label{eq:aux}
        \ud(\sigma^{(k,K)}_t)^i=\eta(\sigma^{(k,K)}_t,k,K)\Big(\frac12(\coth(\sigma^{(k,K)}_t)^i+\sum_{j\neq i}\frac{\sinh(\sigma^{(k,K)}_t)^i}{\cosh(\sigma^{(k,K)}_t)^i-\cosh(\sigma^{(k,K)}_t)^j})\ud t+\sqrt{\frac2\beta}\ud W^i_t\Big)\,,
    \end{align}
    where $\eta(\sigma,k,K)=h(\frac{-S(\sigma)}{k})h(\frac{N(\sigma)}{K})$ and $h$ is a smooth function with
    \begin{align}
        h(x)=\begin{cases}
            1& x\leq 1\\
            0& x\geq 2
        \end{cases}\,.
    \end{align}
    Thus, when $S(\sigma)>-k$, $N(\sigma)<K$, the SDE \eqref{eq:aux} and the SDE \eqref{eq:siegel-eig2} coincides.

Now define two stopping times 
\begin{align}
    \tau_{k,K}:=&\inf\{t>0|S(\sigma^{(k,K)}_t)\leq -k\}\,,\\
    T_{k,K}:=&\inf\{ t>0|N(\sigma^{(k,K)}_t)\geq K\}\,.
\end{align}
Notice that when $S(\sigma)>-k$ and $N(\sigma)<K$, $\sigma^i, S(\sigma),\frac{\partial S}{\partial \sigma^i}$ are all bounded, where $\frac{\partial S}{\partial \sigma^i}=\ee^{-S}\frac{\partial (\ee^S)}{\partial \sigma^i}$ is bounded because $\ee^S$ is smooth in $\{\sigma\in\mathbb R^n|N(\sigma)\leq K\}$. We now have the following two Lemmas for $T_{k,K}$ and $\tau_{k,K}$.

\begin{lemma}\label{lem:TK}
    For $t>0, k>-S(\sigma_0)$, $\lim_{K\to\infty}\mathbb P[T_{k,K}<t\wedge \tau_{k,K}]=0$ uniformly in $k$.
\end{lemma}

\begin{lemma}\label{lem:tauk}
    For $K>N(\sigma_0),t>0$, $\lim_{k\to\infty}\mathbb P[\tau_{k,K}<T_{k,K}\wedge t]=0$.
\end{lemma}

\begin{corollary}\label{col:finite}
    Given $t>0$, almost surely there exists $k_*>-S(\sigma_0),K_*>N(\sigma_0)$ such that $t<\tau_{k_*,K_*}\wedge T_{k_*,K_*}$.
\end{corollary}
\begin{proof}
    For each $n\in\mathbb N$, there exists $K_n$ such that 
    \begin{align}
    \mathbb P[T_{k,K_n}<t\wedge \tau_{k,K_n}]<2^{-n}\,,    
    \end{align}
    for all $k>-S(\sigma_0)$ by Lemma \ref{lem:TK}. There also exists $k_n$ such that 
    \begin{align}
    \mathbb P[\tau_{k_n,K_n}<t\wedge T_{k_n,K_n}]<2^{-n}\,,
    \end{align}
    by Lemma \ref{lem:tauk}.

    Since $\mathbb P[t>T_{k_n,K_n}\wedge \tau_{k_n,K_n}]\leq \mathbb P[\tau_{k_n,K_n}<t\wedge T_{k_n,K_n}]+\mathbb P[T_{k_n,K_n}<t\wedge \tau_{k_n,K_n}]$
    \begin{align}
        &\sum_{n=1}^\infty \mathbb P[t>T_{k_n,K_n}\wedge \tau_{k_n,K_n}]<\infty
    \end{align}
    by the Borel-Cantelli Lemma, with probability 1, $\{t>T_{k_n,K_n}\wedge \tau_{k_n,K_n}\}_{n=1}^\infty$ happens finitely often. Hence, there exist an integer $m$ such that $t\leq T_{k_m,K_m}\wedge \tau_{k_m,K_m}$.
\end{proof}

Corollary \ref{col:finite} leads to proof of Theorem \ref{thm:non-collision} directly:
\begin{proof}[Proof of Theorem \ref{thm:non-collision}]
    For all $k>-S(\sigma_0),K>N(\sigma_0)$, SDE \eqref{eq:aux} has a unique strong solution $\sigma^{(k,K)}_t$ because coefficients of drift and diffusion are all Lipschitz. 

    For each $t>0$, almost surely there exists $k_*,K_*$ such that $t<\tau_{k_*,K_*}\wedge T_{k_*,K_*}$, we define  $\sigma_t=\sigma^{(k_*,K_*)}_t$. By uniqueness of solution of SDE \eqref{eq:aux}, such definition is well-defined and does not depend on the choice of $k_*,K_*$, because $\sigma^{(k,K)}_t=\sigma^{(k',K')}_t$ when $t\leq \tau_{k,K}\wedge T_{k,K}$ and $t\leq \tau_{k',K'}\wedge T_{k',K'}$.

    It is clear that $\sigma_t$ is a solution to SDE \eqref{eq:siegel-eig2}, and uniqueness follows the uniqueness of $\sigma^{(k,K)}_t$ of SDE \eqref{eq:aux}.

    The finiteness of $S(\sigma_t)$ comes from that $t<\tau_{k_*,K_*}$ which means $S(\sigma_t)>-k_*$.
\end{proof}

The proof of Lemma \ref{lem:TK} and Lemma \ref{lem:tauk} are as follows:
\begin{proof}[Proof of Lemma \ref{lem:TK}]
    Consider $N^{k,K}_{t}=N(\sigma^{(k,K)}_{t})$, using It\^o's formula, when $t\leq \tau_{k,K}\wedge T_{k,K}$, we have 
    \begin{align}
        N^{k,K}_{t}
        =&N(\sigma_0)+\sqrt{\frac2\beta}\int_0^{t}\sum_{i=1}^n\frac{\sinh(\sigma^{(k,K)}_s)^i}{\sum_{l=1}^n\cosh(\sigma^{k,K}_s)^l}\ud W^i_s\\
        &+\int_0^{t}\big(\frac n2+\frac1\beta-\frac1\beta \frac{\sum_{i=1}^n\sinh^2(\sigma^{(k,K)}_s)^i}{(\sum_{i=1}^n\cosh(\sigma^{(k,K)}_s)^i)^2}\big)\ud s\nonumber\,.
    \end{align}

    As $\sigma^{(k,K)}_s$ is bounded for $0\leq s\leq t\wedge \tau_{k,K}\wedge T_{k,K}$, each term on the R.H.S are integrable. Take expectation of $N_{t\wedge \tau_{k,K}\wedge T_{k,K}}$:
    \begin{align}
        K\mathbb P[T_{k,K}< t\wedge \tau_{k,K}]\leq &\mathbb E[N_{t\wedge\tau_{k,K}\wedge T_{k,K}}]\\
        \leq& N(\sigma_0)+(\frac n2+\frac1\beta)\mathbb E[\int_0^{t\wedge\tau_{k,K}\wedge T_{k,K}}\ud s]\nonumber\\
        =&N(\sigma_0)+(\frac n2+\frac1\beta)\mathbb E[\int_0^{t\wedge\tau_{k,K}\wedge T_{k,K}}\ud s]\nonumber\\
        \leq &N(\sigma_0)+(\frac n2+\frac1\beta)t\nonumber\,,
    \end{align}
    so $\lim_{K\to \infty}\mathbb P[T_{k,K}<t\wedge \tau_{k,K}]=0$ uniformly in $k$.
    
\end{proof}

\begin{proof}[Proof of Lemma \ref{lem:tauk}]

    We show that when $\beta\geq 2$, $S(\sigma^{(k,K)}_{t\wedge \tau_{k,K}\wedge T_{k,K}})$ is a submartingale. We start with the following observation:
    \begin{align}
        \Delta S+|\nabla S|^2=c_n\,,
    \end{align}
    where $\Delta=\sum_{i=1}^n\frac{\partial^2}{\partial\sigma^i\partial\sigma^i}$ is the Laplacian in Euclidean sense, $|\nabla S|^2=\sum_{i=1}^n(\frac{\partial S}{\partial \sigma^i})^2$, and $c_n=\frac{n(n+1)(2n+1)}6>0$.

    When $t\leq \tau_{k,K}\wedge T_{k,K}$, using It\^o's formula and SDE \eqref{eq:siegel-eig2}, we have 
    \begin{align}
        S(\sigma^{(k,K)}_t)=&S(\sigma_0)+\int_0^t(\frac12|\nabla S|^2+\frac1\beta\Delta S)\ud s+\int_0^t\sum_{i=1}^n\sqrt{\frac2\beta}\frac{\partial S}{\partial \sigma^i}\ud W^i_s\\
        =&S(\sigma_0)+\int_0^t(\frac{c_n}\beta+(\frac12-\frac1\beta)|\nabla S|^2)\ud s+\int_0^t\sum_{i=1}^n\sqrt{\frac2\beta}\frac{\partial S}{\partial \sigma^i}\ud W^i_s\nonumber
    \end{align}

    When $0\leq s\leq t\wedge \tau_{k,K}\wedge T_{k,K}$, every integral above are integrable. Using the condition $\beta\geq 2$, we have 
    \begin{align}
        S(\sigma_0)<\mathbb E[S_{t\wedge \tau_{k,K}\wedge T_{k,K}}]
    \end{align}
    
    Also,
    $S_{t\wedge \tau_k\wedge T_K}$ is bounded above by a constant $C_K$ depending on $K$, we have 
    \begin{align}
        S_0<\mathbb E[S_{t\wedge \tau_{k,K}\wedge T_{k,K}}]<-k\mathbb P[\tau_{k,K}<t\wedge T_{k,K}]+C_K(1-\mathbb P[\tau_{k,K}<t\wedge T_{k,K}])\,,
    \end{align}
    which leads to $\lim_{k\to\infty}\mathbb P[\tau_{k,K}<t\wedge T_{k,K}]=0$ for given $t>0,K>N(\sigma_0)$.

\end{proof}

\bibliographystyle{amsplain}
\bibliography{siegel}
\end{document}